\documentclass[runningheads,envcountsame]{llncs}
\usepackage[utf8]{inputenc}
\usepackage{graphicx}
\usepackage{amsmath,amssymb}
\usepackage[mathlines,pagewise]{lineno}
\usepackage{listings}
\usepackage[draft]{varioref}
\usepackage{url}
\usepackage{hyperref}
\usepackage{todonotes}
\usepackage{tabularx}

\renewcommand{\epsilon}{\varepsilon}

\newcommand{\MF}{\mathcal{MF}}

\newcommand{\DC}{\mathcal{DC}}
\newcommand{\DOM}{\mathcal{D}}

\newcommand{\Bal}{\mathrm{Bal}}

\renewcommand{\epsilon}{\varepsilon}

\begin{document}

\title{Digital Convexity and Combinatorics on Words}
\titlerunning{Digital Convexity and Combinatorics on Words}

\author{Alessandro De Luca \inst{1}\orcidID{0000-0003-1704-773X} \and Gabriele Fici \inst{2}\orcidID{0000-0002-3536-327X}  \and Andrea Frosini \inst{3}\orcidID{0000-0001-7210-2231} }

\authorrunning{A.~De Luca \and G.~Fici \and A.~Frosini}

\institute{DIETI, Università di Napoli Federico II, Italy\\ \email{alessandro.deluca@unina.it} \and Dipartimento di Matematica e Informatica, Università di Palermo, Italy \\
\email{gabriele.fici@unipa.it} \and Dipartimento di Matematica e Informatica, Università di Firenze, Italy \\
\email{andrea.frosini@unifi.it}}

\maketitle

\begin{abstract}
An upward (resp.~downward) digitally convex  word is a binary word that best approximates from below (resp.~from above) an upward (resp.~downward) convex curve in the plane. We study these words from the combinatorial point of view, formalizing their geometric properties and highlighting connections with Christoffel words and finite Sturmian words. In particular, we study from the combinatorial perspective the operations of inflation and deflation on digitally convex words.
\end{abstract}

\section{Introduction}

Combinatorics on words and digital geometry have a long history of interactions. In particular, digital approximations of lines in the plane have a natural encoding as binary words. In this paper, we fix the binary alphabet $\{0,1\}$, where $0$ (resp.~$1$) represents a horizontal (resp.~vertical) unary step in the grid $\mathbb{N}\times \mathbb{N}$.

For example, there is a clear and well understood correspondence between finite words and approximations of digital segments, given by Christoffel words. For every pair of positive integers $(a,b)$ there is one lower Christoffel word with $a$ occurrences of $0$ and $b$ occurrences of $1$, which is the digital approximation from below  of the Euclidean segment from $(0,0)$ to $(a,b)$. Christoffel words are detailed in many classical references \cite{DBLP:journals/tcs/Reutenauer15,LothaireAlg,JTNB_1993}. 
Relationships between Christoffel words and convex digital shapes have been investigated in several papers~\cite{DBLP:conf/acpr/TarsissiCKR19,DBLP:conf/cwords/DulioFRTV17,DBLP:journals/pr/BrlekLPR09,DBLP:journals/jco/DulioFRTV22}.

In this paper, we focus on digital approximations of convex curves in the plane. A binary word is called (upward) digitally convex if it is  the digital approximation from below of a convex curve joining $(0,0)$ to $(a,b)$ and contained in the rectangle whose opposite vertices are $(0,0)$ and $(a,b)$. Every such word  has $a$ occurrences of $0$ and $b$ occurrences of $1$. In particular, all the digitally convex words with $a$ occurrences of $0$ and $b$ occurrences of $1$ lie between the lower Christoffel word with $a$ occurrences of $0$ and $b$ occurrences of $1$ and the word $1^b0^a$, the upper left contour of the rectangle.


In~\cite{DBLP:journals/pr/BrlekLPR09}, the authors gave a purely combinatorial characterization of digitally convex words: a word $w$ is digitally convex if and only if all the Lyndon words in the Lyndon factorization of $w$ are balanced (hence primitive lower Christoffel words). A primitive word (i.e., a word that cannot be written as a concatenation of copies of a shorter word) over $\{0,1\}$ is a Lyndon word if it is lexicographically smaller (for the order $0<1$) than all its proper suffixes. Every nonempty word $w$  has a unique factorization in non-increasing Lyndon words, which is called the Lyndon factorization of $w$. 

A binary word $w$ is balanced (or Sturmian) if for every length $i$, the difference between the number of occurrences of $0$s and $1$s in any two factors of length $i$ of $w$ is at most $1$. 
Since a  word over $\{0,1\}$ is a primitive lower Christoffel word if and only if it  is balanced and Lyndon, the previous characterization can be seen as a natural decomposition of  digitally convex words in straight segments.

In this paper, we give combinatorial results on digitally convex words that shed light on their combinatorial properties and are related to the geometry of convex curves that these words approximate.

The paper is organized as follows: In Section~\ref{sec:prelim} we fix the notation and give preliminary results on words, and in particular on Christoffel words, that we will need in the sequel; in Section~\ref{sec:digconv} we recall the notion of digitally convex words and provide results and characterizations. Finally, in Section~\ref{sec:infldefl} we study from the combinatorial perspective the operations of inflation and deflation on digitally convex words. In particular, we show that any digitally convex word with $a$ occurrences of $0$ and $b$ occurrences of $1$ can be obtained with a sequence of applications of the inflation operation from the lower Christoffel word with the same number of $0$s and $1$, but also with a sequence of applications of the deflation operation from the word $1^b0^a$.

\section{Preliminaries}\label{sec:prelim}

 A \emph{word} over a finite alphabet $\Sigma$ is a concatenation of letters from $\Sigma$. The \emph{length} of a word $w$ is denoted by $|w|$. The empty word $\varepsilon$ has length $0$. The set of all words 
 over  $\Sigma$ is denoted  $\Sigma^*$
 and is a free monoid
 with respect to concatenation. 
 For a letter $x\in\Sigma$, $|w|_x$ denotes the number of occurrences of $x$ in $w$. If $\Sigma=\{x_1,\ldots,x_\sigma\}$ is ordered, the vector $(|w|_{x_1},\ldots,|w|_{x_\sigma})$ is the \emph{Parikh vector} of $w$.

Let $w=uv$, with $u,v\in \Sigma^*$. We say that $u$ is a \emph{prefix} of $w$ and that $v$ is a \emph{suffix} of $w$. A \emph{factor} of $w$ is a prefix of a suffix (or, equivalently, a suffix of a prefix) of $w$. 
If $w=w_1w_2\cdots w_n$, with $w_k\in\Sigma$ for all $k$, we let $w[i..j]$ denote the nonempty factor $w_iw_{i+1}\cdots w_j$ of $w$, whenever $1\leq i\leq j\leq n$.
A factor $u$ of a word $w\neq u$ is a \emph{border} of $w$ if $u$ is both a prefix and a suffix of $w$; in this case, $w$ has \emph{period} $|w|-|u|$. A word $w$ is \emph{unbordered} if its longest border is $\varepsilon$; i.e., if its smallest period is $|w|$. 
For a word $w$, the \emph{$n$-th power} of $w$ is the word $w^n$ obtained by concatenating $n$ copies of $w$. 

Two words $w$ and $w'$ are \emph{conjugates} if $w=uv$ and $w'=vu$ for some words $u$ and $v$. The conjugacy class of a word $w$ contains $|w|$ elements if and only if $w$ is \emph{primitive}, i.e., $w=v^n$ for some $v$ implies $n=1$.

A nonempty word $w=w_1w_2\cdots w_n$, $w_k\in \Sigma$ for all $k$, is a \emph{palindrome} if  it coincides with its \emph{reversal} $w^R=w_nw_{n-1}\cdots w_1$. The empty word is also assumed to be a palindrome. 

The following proposition, whose proof is straightforward, is well known (cf.~\cite{DELUCA1982207,DBLP:journals/ijfcs/BrlekHNR04}). 

\begin{proposition}\label{prop:2pal}
 A word is a conjugate of its reversal if and only if it is a concatenation of two palindromes. Moreover, these two palindromes are uniquely determined if and only if the word is primitive.
\end{proposition}
 


\subsection{Lyndon Words}


 A primitive word over an ordered alphabet is a \emph{Lyndon word} if it is lexicographically smaller than all its conjugates (or, equivalently, lexicographically smaller than all its proper suffixes).

In the binary case, we have the class of Lyndon words for the order $0<1$ and the class of Lyndon words for the reversed order $1<0$. Whenever not differently specified, we assume the order $0<1$.

The following result, originally due to Chen, Fox and Lyndon  (1958) is a classical result in combinatorics on words (see~\cite{Lot01}):

\begin{theorem}\label{thm:CLF}
 Every word $w$ has a unique non-increasing factorization in Lyndon words.
\end{theorem}
The factorization from the previous theorem is called the \emph{Lyndon factorization} (or sometimes the Chen--Fox--Lyndon factorization) of $w$.


\begin{example}\label{ex:dc}
Let $w = 0101001001$. The Lyndon factorization of $w$ is \[01\cdot 01\cdot 001\cdot 001.\]
\end{example}

\begin{example}\label{ex:dc2}
Let $w = 1100$. The Lyndon factorization of $w$ is \[1\cdot 1\cdot 0\cdot 0.\]
\end{example}

A way to obtain the Lyndon factorization is by using the following:

\begin{property}\label{propLyn}
If $u$ and $v$ are Lyndon words, then  $u<v$ if and only if $uv$ is a Lyndon word.
\end{property}

Starting from the trivial factorization $w=w_1\cdots w_n$, with $|w_i|=1$ for every $i$, one repeatedly applies Property~\ref{propLyn} until the factors occur in non-increasing order.

Every Lyndon word $w$ of length $|w|\geq 2$ has a \emph{standard factorization} $w=uv$, where $v$ is the lexicographically least proper suffix of $w$ (or, equivalently, the longest proper suffix of $w$ that is a Lyndon word), see~\cite{Lot01}.

\subsection{Balanced Words}

\begin{definition}
 A word over $\{0,1\}$ is \emph{balanced} if the number of $0$s (or, equivalently, $1$s) in every two factors of the same length differs by at most $1$.
\end{definition}

We let $\Bal$ denote the set of balanced words over the alphabet $\{0,1\}$. Balanced words are precisely the finite factors of infinite Sturmian words~\cite{LothaireAlg}. In fact, balance is a \emph{factorial property}, i.e., every factor of a balanced word is itself balanced.

The number of  balanced words of  length $n$ is known~\cite{Mignosi_On_the_number,Lipatov198267} to be 
\[1+\sum_{k=1}^n (n-k+1)\phi(k).\]
In particular, it grows polynomially in $n$.



 \subsection{Christoffel Words}

We now define Christoffel words and recall their main properties. We point the reader to the classical references~\cite{Book08,ReutenauerMarkoff,LothaireAlg,Be07} for more details on Christoffel words.

Christoffel words are powers of balanced Lyndon words. But let us introduce them from the point of view of digital geometry.

In what follows, we fix the alphabet $\{0,1\}$ and represent a word over $\{0,1\}$ as a path in $\mathbb{N}\times \mathbb{N}$ starting at $(0,0)$, where $0$ (resp., $1$) encodes a horizontal (resp., vertical) unit segment.
For every pair of nonnegative integers $(a,b)$ (not both $0$), every word $w$ that encodes a path from $(0,0)$ to $(a,b)$ must have exactly $a$ zeroes and $b$ ones, i.e., it has Parikh vector $(a,b)=(|w|_0,|w|_1)$. 
We define the \emph{slope} of a word $w$ with Parikh vector $(a,b)$ as the rational number $b/a$ if $a\neq 0$, or $\infty$ otherwise.

\begin{definition}
Given a pair of nonnegative integers $(a,b)$ (not both $0$), the \emph{lower (resp.,~upper) Christoffel word} $w_{a,b}$ (resp.,~$W_{a,b}$), of slope $b/a$, is the word encoding the path from $(0,0)$ to $(a,b)$ that is closest from below (resp.,~from above) to the Euclidean segment, without crossing it. 
 \end{definition}

 In other words, $w_{a,b}$ (resp.~$W_{a,b}$) is the digital approximation from below (resp., from above) of the Euclidean segment joining $(0,0)$ to $(a,b)$.
 For example, the lower Christoffel word $w_{7,4}$ is the word $00100100101$ (see~Fig.~\ref{fig:chr}).
 
 By construction, we have that the upper Christoffel word $W_{a,b}$ is the reversal of the lower Christoffel word $w_{a,b}$ (and the two words are conjugates, see below).

 Notice that by definition we allow $a$ or $b$ (but not both) to be $0$, that is, we have $w_{n,0}=W_{n,0}=0^n$ and $w_{0,n}=W_{0,n}=1^n$, so that the powers of words of length $1$ are both lower and upper Christoffel words.

 \begin{figure}[ht]
     \centering
     \includegraphics[width=50mm]{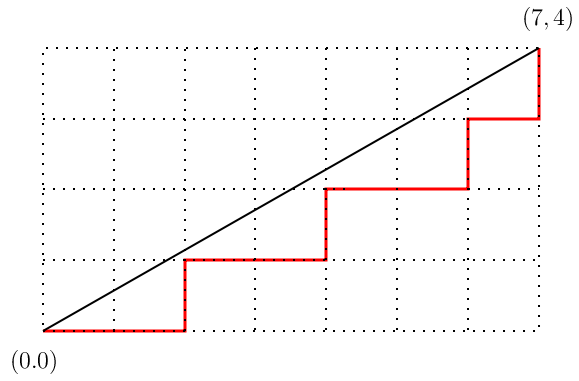}
     \caption{The lower Christoffel word $w_{7,4}=00100100101$. The upper Christoffel word $W_{7,4}=10100100100$ is the reversal of $w_{7,4}$.}
     \label{fig:chr}
 \end{figure} 

If $a$ and $b$ are coprime, the Christoffel words $w_{a,b}$ and $W_{a,b}$ do not intersect the Euclidean segment joining $(0,0)$ to $(a,b)$ (other than at the end points) and are primitive words. If instead $a=n\alpha$ and $b=n\beta$ for some $n>1$, then $w_{a,b}=(w_{\alpha,\beta})^n$ (resp.,~$W_{a,b}=(W_{\alpha,\beta})^n$).  Hence, $w_{a,b}$ (resp.,~$W_{a,b}$) is primitive if and only if $a$ and $b$ are coprime.

Therefore, letting $\phi$ denote the Euler totient function, for every $n> 1$ there are $2\phi(n)$ primitive Christoffel words of length $n$ (in particular, there are $\phi(n)$ primitive lower Christoffel words and $\phi(n)$ primitive upper Christoffel words).

We now give a fundamental definition for the combinatorics of Christoffel words.

\begin{definition}
 A \emph{central word} is a word that has coprime periods $p$ and $q$ and length $p+q-2$.
\end{definition}

It is well known that central words are binary palindromes~\cite{DBLP:journals/tcs/LucaM94,LothaireAlg}. The following theorem gives a structural characterization (see~\cite{DBLP:journals/tcs/Luca97a}).

\begin{theorem}[Combinatorial Structure of Central Words]\label{thm:CSCW}
 A word $w$ is central if and only if it is a power of a letter or there exist palindromes $P$ and $Q$ such that $w=P01Q=Q10P$.
 Moreover,
\begin{itemize}
\item $P$ and $Q$ are central words;
\item $|P|+2$ and $|Q|+2$ are coprime and $w$ has periods $|P|+2$ and $|Q|+2$;
\item if $|P|<|Q|$, $Q$ is the longest palindromic (proper) suffix of $w$.
\end{itemize}
\end{theorem}

\begin{example}
The word $w=010010$ is a central word, since it has periods $3$ and $5$ and length $6=3+5-2$. We have $010010=010 \cdot 01 \cdot 0=0 \cdot 10 \cdot 010$, so that $P=010$ and $Q=0$.
\end{example}
 
\begin{proposition}[\cite{Pirillo2001}]\label{prop:paldec}
A word $C$ is central if and only if the words $0C1$ and $1C0$ are conjugates.
\end{proposition}
 

\begin{proposition}[\cite{DBLP:journals/tcs/BerstelL97}]
 A word is a primitive lower (resp.,~upper) Christoffel  word if and only if it has length $1$ or it is of the form $0C1$ (resp.,~$1C0$) where $C$ is a central word.
\end{proposition}

\begin{example}
 Let $a=7$ and $b=4$. We have  $w_{7,4}=00100100101=0\cdot 010010010 \cdot 1$, where $~C=010010010$ is a central word since it has periods $3$ and $8$ and length $9=3+8-2$; see Fig.~\ref{fig:central}.
\end{example}

So, central words are the ``central'' factors of primitive Christoffel  words of length~$\geq 2$. 
A geometric interpretation of the central word $C$ is the following: it encodes the intersections of the Euclidean segment joining $(0,0)$ to $(a,b)$ ($0$ for a vertical intersection and $1$ for a horizontal intersection). In other words, the word $C$ is the cutting sequence of the Euclidean segment joining $(0,0)$ to $(a,b)$; see again Fig.~\ref{fig:central}.

 \begin{figure}[ht]
     \centering
    \includegraphics[width=50mm]{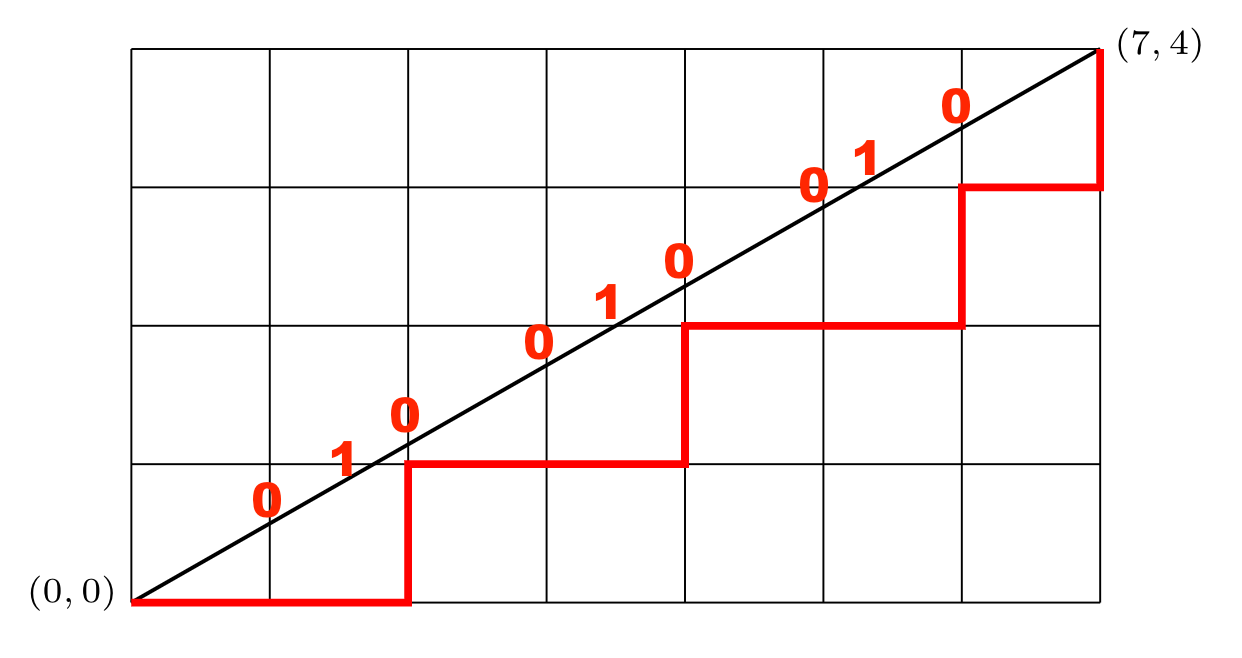}
     \caption{The central word $C=010010010$ is the central factor of the primitive lower Christoffel  word $w_{7,4}=0C1$. It encodes the intersections of the Euclidean segment joining $(0,0)$ and $(7,4)$ with the grid ($0$ for a vertical intersection and $1$ for a horizontal intersection).}
     \label{fig:central}
 \end{figure}

\begin{proposition}[\cite{DBLP:journals/ejc/BertheLR08}]\label{prop:multinv}
Let $w_{a,b}=0C1$ be a primitive lower Christoffel word. The central word $C$ has periods $a'$ and $b'$, the multiplicative inverses of $a$ and $b$ modulo $a+b$, respectively.
\end{proposition}



\begin{theorem}[\cite{DBLP:journals/tcs/BerstelL97}]\label{thm:Lyndon}
Let $w$ be a  word over $\{0,1\}$. Then $w$ is a primitive (lower or upper) Christoffel word if and only if it is balanced and unbordered.
\end{theorem}

In particular, the set of primitive lower Christoffel  words is precisely the set of balanced Lyndon words over the alphabet $\{0,1\}$  for the order $0<1$.

\begin{proposition}[\cite{JTNB_1993}]\label{prop:BL}
 For every coprime positive integers $a,b$, the primitive lower Christoffel word $w_{a,b}$ is the greatest (for the lexicographic order) word among all Lyndon words having Parikh vector $(a,b)$.
\end{proposition}

\begin{proposition}[\cite{DELUCA2024114935}]\label{cor:surprising}
 For all coprime positive integers $a,b$, the lower Christoffel word $w_{a,b}$ is the smallest (in the lexicographic order)  word among all balanced words having Parikh vector $(a,b)$.
\end{proposition}

\begin{theorem}[\cite{JTNB_1993}]
    \label{thm:slopes}
    Let $(a,b)$ and $(c,d)$ be pairs of nonnegative integers, both distinct from $(0,0)$, such that $b/a\neq d/c$. Then
    \[w_{a,b}<w_{c,d}\iff\frac ba<\frac dc.\]
\end{theorem}

Since primitive lower Christoffel words are Lyndon words, every primitive lower Christoffel word of length $|w|\geq 2$ has a standard factorization. 

On the other hand, by Proposition~\ref{prop:paldec},
a primitive lower Christoffel word is a conjugate of its reversal (the corresponding upper Christoffel word); hence by Proposition~\ref{prop:2pal}, every primitive lower Christoffel word of length $|w|\geq 2$ has a unique \emph{palindromic factorization}.

Let $w_{a,b}=0C1$ be a primitive lower Christoffel word, so that $a$ and $b$ are coprime integers. If the central word  $C$ is not a power of a single letter, then by Theorem~\ref{thm:CSCW} there exist central words $P$ and $Q$ such that $C=P01Q=Q10P$ so that $w_{a,b}=0C1=0P0\cdot 1Q1=0Q1\cdot 0P1$. 

Hence, we have the following factorizations:

\begin{enumerate}
 \item $0C1=0P0\cdot 1Q1$ (palindromic factorization);
 \item $0C1=0Q1\cdot 0P1$ (standard factorization).
\end{enumerate}


If instead $C=0^n$ (the case $C=1^n$ is analogous)  we have:
\begin{enumerate}
 \item $0C1=0^{n+1}\cdot 1$ (palindromic factorization);
 \item $0C1=0\cdot 0^n1$ (standard factorization).
\end{enumerate}



From the geometric point of view, the standard factorization divides $w_{a,b}$ in two shorter Christoffel words, and it determines the point $S$ of the encoded path that is \emph{closest} to the Euclidean segment from $(0,0)$ to $(a,b)$; the palindromic factorization, instead, divides $w_{a,b}$ in two palindromes and determines the point $S'$ that is \emph{farthest} from the Euclidean segment (see~\cite{DBLP:journals/ita/BorelR06,LamaPhD}). An example is given in~Fig.~\ref{fig:dec}.  

\begin{figure}[ht]
     \centering
 \includegraphics[width=120mm]{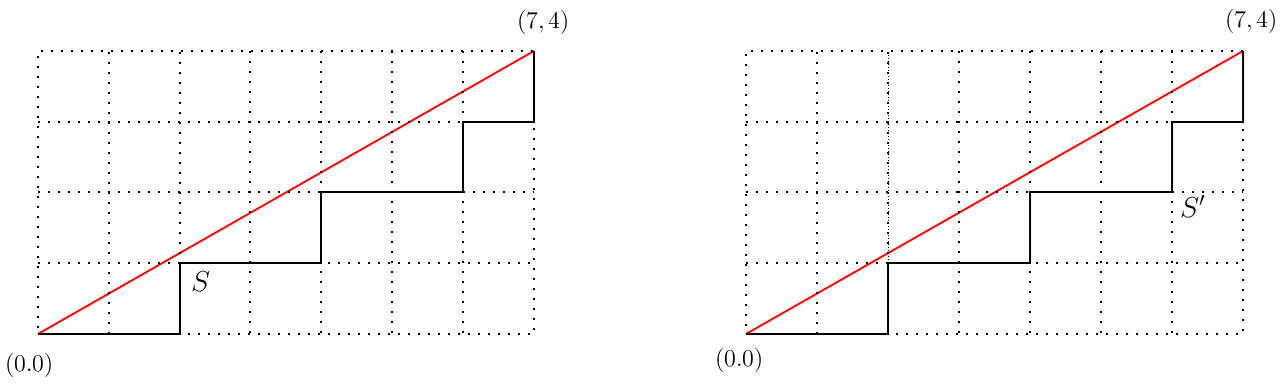}    
 \caption{The standard factorization $0Q1\cdot 0P1 = 001 \cdot 00100101$ (left) and the palindromic factorization $0P0\cdot 1Q1 = 00100100\cdot 101$ (right) of the lower Christoffel word $w_{7,4}$. The point $S$ determined by the standard factorization is the closest to the Euclidean segment, while the point $S'$ determined by the palindromic factorization is the farthest.}
     \label{fig:dec}
 \end{figure}


\section{Digitally Convex Words}\label{sec:digconv}

We now introduce the main object of study of this paper.

\begin{definition}
    A binary word with $a$ occurrences of $0$ and $b$ occurrences of $1$ is an \emph{upward (resp.~downward) digitally convex word} if it is the best approximation from below (resp.~above) of an upward (resp.~downward) convex curve that joins $(0,0)$ and $(a,b)$ and is contained in the rectangle whose opposed vertices are $(0,0)$ and $(a,b)$. 
\end{definition}

For example, the word $w = 0101001001$ of Example~\ref{ex:dc} is (upward) digitally convex, as shown in Fig.~\ref{fig:digconv}.

\begin{figure}[ht]
     \centering
 \includegraphics[width=50mm]{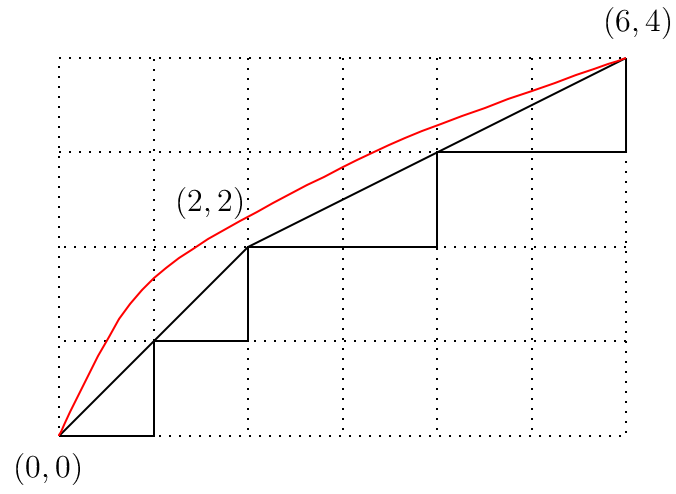}    
 \caption{The (upward) digitally convex word $w = 0101001001$ and its decomposition in two Christoffel words $0101$ and $001001$. In red, one of the possible approximated upward convex curves that join $(0,0)$ and $(6,4)$.}
     \label{fig:digconv}
 \end{figure}

\begin{proposition} \label{prop:convud}
We have:
    \begin{enumerate}
        \item  A word is upward (resp.~downward) digitally convex if and only if its reverse is downward (resp.~upward) digitally convex;
        \item  A word is upward (resp.~downward) digitally convex if and only if its binary complement is downward (resp.~upward) digitally convex.
    \end{enumerate}
\end{proposition}

\begin{proof}
    By symmetry.
\end{proof}

Geometrically, fixed $a$ and $b$ positive integers, the upward digitally convex words with Parikh vector $(a,b)$ are all above the Christoffel word $w_{a,b}$ and below the word $1^b0^a$. And of course the downward digitally convex words with Parikh vector $(a,b)$ are all below the Christoffel word $W_{a,b}$ and above the word $0^a1^b$. This will be proved formally in Theorem \ref{thm:dura2}.

Every balanced word is, as we will see, both upward and downward digitally convex. On the contrary, digitally convex words are not necessarily balanced, as  shown by the word $w = 0101001001$ of Example~\ref{ex:dc}. 
We have the following characterization, whose proof will be given later:

\begin{proposition}\label{prop:sym}
    Let $w$ be a binary word, $\widetilde{w}$ its reverse, and $\overline{w}$ its binary complement. The following are equivalent:
\begin{enumerate}
    \item $w$ is balanced;
    \item $w$ and $\widetilde{w}$ are both upward  (resp.~downward) digitally convex;
    \item $w$ and $\overline{w}$ are both upward  (resp.~downward) digitally convex;
    \item $w$ is upward digitally convex and downward digitally convex.
\end{enumerate}
\end{proposition}

So, the words that are simultaneously upward and downward digitally convex are precisely the balanced words. However, there are digitally convex words that lie between the lower and the upper Christoffel word but are not balanced (for instance the word  $0101000$).

The Lyndon factorization can be used to characterize digitally convex words, as it was shown by Brlek et al.~in the following

\begin{theorem}[\cite{DBLP:journals/pr/BrlekLPR09}]\label{thm:digcon}
A word $w$ is upward digitally convex if and only if all the Lyndon words in the Lyndon factorization of $w$ are balanced (hence primitive lower Christoffel words).
\end{theorem}

An analogous result characterizes downward digitally convex words. In fact, by a simple symmetry argument, we have that a binary word is downward digitally convex if and only if  all the Lyndon words for the order $1<0$ in the Lyndon factorization of $w$  are balanced (hence primitive upper Christoffel words).

\begin{remark}
The union of upward and downward digitally convex words is a factorial 
language, that is also closed under reversal and binary complement, and includes the language of balanced words.
\end{remark}

Because of the symmetries between upward and downward digitally convex words, from now on we will focus on upward digitally convex words only, which we will simply call digitally convex words. Analogous results hold of course for downward digitally convex words.

\subsection{Minimal Forbidden Words}

Minimal forbidden words are a useful tool for characterizing factorial languages.

\begin{definition}
Let $L$ be a factorial language. A word $w$ is a minimal forbidden word for $L$ if $w$ does not belong to $L$ but all proper factors of $w$ do. 
\end{definition}

Let $\MF(L)$ denote the set of minimal forbidden words of the language $L$. A word $w=xvy$, $x,y$ letters, belongs to $\MF(L)$ if and only if
\begin{enumerate}
 \item $xvy\not\in L$;
 \item $xv,vy\in L$.
\end{enumerate}

A special case is when $L$ is the set of factors of a single word $w$. In this case we talk of minimal forbidden words of the word $w$.

 \begin{example}
 Let $w=01001$. The minimal forbidden words (MFWs) of  $w$ are:
 $$\MF(w)=\{000, 0010, 101, 11\}.$$
 \end{example}

\begin{theorem}[\cite{DBLP:conf/birthday/MignosiRS99}]
 There is a bijection between factorial languages and their sets of minimal forbidden words.
\end{theorem}

As a consequence, $\MF(L)$ uniquely determines $L$.

Let $\Bal$ be the (factorial) language of balanced words over $\{0,1\}$. The next theorem gives a characterization of the minimal forbidden words for the language of binary balanced words.

\begin{theorem}[\cite{DBLP:journals/jcss/Fici14}]
 $\MF(\Bal)=\{yvx \mid \{x,y\}=\{0,1\}$, $xvy$ is a non-primitive Christoffel~word$\}$.
\end{theorem}

\begin{example}
The word  $000101$ is not balanced, but all its proper factors are. Indeed, $100100$ is the square of the primitive upper Christoffel word $100$.

\end{example}

\begin{corollary}[\cite{DBLP:journals/jcss/Fici14}]\label{cor:Fici}
 For every $n>0$, there are exactly $n-\phi(n)-1$ words of length $n$ in  $\MF(\Bal)$ that start with $0$, and they are all Lyndon words. Here $\phi$ is the Euler totient function.
\end{corollary}

In 2011, Proven\c{c}al~\cite{DBLP:journals/tcs/Provencal11} studied the minimal forbidden words of the set $\DC$ of digitally convex words.
In fact, the set of digitally convex words is a factorial language. In his paper, Proven\c{c}al  attributes this result to a private communication of C.~Reutenauer, but here we provide a proof for the sake of completeness.

\begin{proposition}\label{prop:digconvfact}
 Every factor of a digitally convex word is digitally convex.
\end{proposition}

\begin{proof}
 It is sufficient to prove that if $wx$ (resp.~$xw$) is digitally convex, $x\in\{0,1\}$, then so is $w$.
 Indeed, let $\ell_1\cdots\ell_k$ be the Lyndon factorization of $w$. Then, using Property~\ref{propLyn}, the Lyndon factorization of $wx$ is either $\ell_1\cdots\ell_k\cdot x$, if $\ell_k\geq x$, or $\ell_1\cdots\ell_i\overline{\ell}$ for some $i$, where $\overline{\ell}=\ell_{i+1}\cdots \ell_kx$,  otherwise. In the former case, $w$ is clearly digitally convex. In the latter case, by hypothesis, the Lyndon word $\overline{\ell}$ is balanced, therefore so are $\ell_{i+1},\ldots,\ell_k$. Hence $\ell_{i+1},\ldots,\ell_k$ are balanced Lyndon words, i.e., primitive lower Christoffel words, whence $w$ is  digitally convex. 
 
 The case $xw$ is analogous.
\qed \end{proof}

As a consequence, a word is digitally convex if and only if all its Lyndon factors are balanced. 

\begin{theorem}[\cite{DBLP:journals/tcs/Provencal11}]
\label{thm:maw-Prov}
$\MF(\DC)=\{u(uv)^kv \mid k\geq 1, uv$ is the standard factorization of a primitive lower Christoffel word$\}$.
\end{theorem}

We now give an alternative description:

\begin{theorem}
\label{thm:maw-dc}
$\MF(\DC)$  is the set of words in  $\MF(\Bal)$ that start with $0$. Hence,
$\MF(\DC)=\{0w1 \mid \mbox{$1w0$ is a non-primitive Christoffel word}\}$. In particular, all words in $\MF(\DC)$ are Lyndon words.
\end{theorem}

\begin{proof}
Let $u=0Q1$, $v=0P1$. Then $0w1=u(uv)^kv
= 0Q1\cdot (0Q1\cdot 0P1)^k\cdot 0P1 
= 0Q1\cdot (0P0\cdot 1Q1)^k\cdot 0P1$.
Therefore, $1w0=1Q1\cdot (0P0\cdot 1Q1)^k\cdot 0P0
= (1Q10P0)^{k+1}$.
The other cases are analogous.
\qed \end{proof}

By Corollary~\ref{cor:Fici}, we have:
\begin{corollary}
$\MF(\DC)(n)=n-1-\phi(n)$. 
\end{corollary}

\begin{proof}[of Proposition \ref{prop:sym}]
Conditions 2–4 are equivalent by Proposition~\ref{prop:convud}. By Theorem~\ref{thm:digcon}, if a word is balanced, then it is upward and downward digitally convex. So it remains to prove that if a word is upward and downward digitally convex then it is balanced. This follows from Theorem~\ref{thm:digcon} and Theorem~\ref{thm:maw-dc}. 
\qed \end{proof}

\subsection{Counting Formula}

Let us now give a formula for the number of digitally convex words of a given length. Let $\DC$ be the set of upward digitally convex words. Every word in $\DC$ either starts with $0$ or it is a power of $1$ concatenated with a word in $\DC$ starting with $0$. Hence, if $\DC_0$ denotes the set of digitally convex words staring with $0$, we have that $\DC(n):=|\DC\cap \{0,1\}^n|=\DC_0(n)\cup \{1^kw \mid w\in\DC_0(n-k), 1\leq k\leq n\}$, so that $|\DC(n)|=\sum_{k=0}^n |\DC_0(k)|$.

\begin{theorem}
 The number of digitally convex words starting with $0$ is given by the Euler transform of the Euler totient function $\phi$:
\[|\DC_0(n)|=\dfrac{1}{n}\sum_{k=1}^n\sum_{d\vert k}d \phi(d)|\DC_0(n-k)|\]
\end{theorem}

\begin{proof}
For every $d\geq 1$, there are precisely $\phi(d)$ primitive Christoffel words of length $d$ that start with $0$, and there is a bijection between digitally convex words of length $n$ starting with $0$ and multisets of primitive Christoffel words that start with $0$ with total length $n$. The formula then follows by applying the Euler transform (see~\cite{EIS}, p.~20--22).
\qed \end{proof}

The previous result is also sketched, using formal power series, in~\cite{DBLP:conf/dgci/BodiniDJM13}. 

The first few values of the sequence $|\DC_0(n)|$ are presented in Table~\ref{tab:dc0}.

\begin{table}[ht]
\begin{center}
\setlength{\tabcolsep}{6.5pt} 
\renewcommand{\arraystretch}{1.3} 
\begin{tabular} {cccccccccccccccccccc}
$n$  &  0   & 1 & 2 & 3 & 4 & 5 & 6  & 7  & 8  & 9  & 10 & 11  & 12  \\ \hline
$|\DC_0(n)|$ & 1 & 1 & 2 & 4 & 7 & 13 & 21 & 37 & 60 & 98 & 157 & 251 & 392
\end{tabular}
\end{center}
\caption{First few values of $|\DC_0(n)|$ (sequence A061255 in~\cite{OEIS}).\label{tab:dc0}}
\end{table}

 Unlike the number of binary balanced words, the number of digitally convex words of length $n$ grows exponentially. For the precise asymptotic refer to entry  A061255 in OEIS~\cite{OEIS}.
Actually, for length $2n$, there are $\Omega(e^{n/2})$ digitally convex words with Parikh vector $(n,n)$. This follows from the fact that for every decomposition of $n$ in distinct parts, there is a digitally convex word with Parikh vector $(n,n)$ having that decomposition as the number of horizontal (or vertical) unit steps. The asymptotic for the sequence of decompositions of $n$ in distinct parts is known (see OEIS:A000009).

\subsection{Dominance Order}

\begin{definition}
    Over $\{0,1\}^n$, we consider the \textit{dominance order} difened by $u\sqsubseteq v$ if for every $1\leq i\leq n$, $|u[1..i]|_1\leq |v[1..i]|_1$ (or, equivalently, $\left |u[1..i]\right |_0\geq |v[1..i]|_0$).
\end{definition}

Notice that the dominance order is a partial order, and that the lexicographic order is a linear extension of it.

 \begin{theorem}\label{thm:dura2}
Let $a,b>0$ and $n=a+b$. For every digitally convex word $u$ with Parikh vector $(a,b)$, and for every $1\leq k\leq n$, we have that $w_{a,b}[1..k]$ is lexicographically smaller than or equal to $u[1..k]$, hence in particular $w_{a,b}[1..k]\sqsubseteq u[1..k]$.
\end{theorem}

\begin{proof}
By contradiction, let $k$ be minimal such that $w_{a,b}[1..k]>u[1..k]$. By minimality, we have $w_{a,b}[1..k]=v1$ and $u[1..k]=v0$ for some word $v$. If $v0$ is not balanced, then the first Christoffel word $u_1$ in the Lyndon factorization of $u$ is a prefix of $v$ and hence of $w_{a,b}$. If $v0$ is balanced instead, then $u_1$ has $v0$ as a prefix. Thus in both cases we have $u_1<w_{a,b}$. By Theorem~\ref{thm:slopes}, this implies that the slope of $u_1$ is at most $b/a$. Since the Lyndon factorization $u=u_1\cdots u_{\ell}$ is lexicographically nonincreasing, we obtain that \emph{all} the Christoffel words $u_1,\ldots,u_\ell$ have slope at most $b/a$. Moreover, since $u\neq w_{a,b}$, at least the last element $u_\ell$ must have slope $<b/a$.
This implies that the sum of the Parikh vectors of $u_1,\ldots, u_n$ (that is, the Parikh vector of $u$) would also be a vector with slope less than $b/a$, a contradiction.
\qed \end{proof} 

\begin{corollary}\label{prop:chrisdc}
For every pair $(a,b)$, the lower Christoffel word $w_{a,b}$ is the smallest (in the lexicographic order) digitally convex word having  Parikh vector $(a,b)$.
\end{corollary}

The previous theorem essentially says that all the paths encoded by upward digitally convex words with Parikh vector $(a,b)$ stay above the path encoded by the lower Christoffel word $w_{a,b}$ (and below the path encoded by $1^b0^a$).

By symmetry, we also have that all the paths encoded by downward digitally convex words with Parikh vector $(a,b)$ stay above the path encoded by $0^a1^b$ and below the upper Christoffel word $W_{a,b}$.

\begin{definition}
To each binary word $w$ it is associated an integer sequence $s_w$ such that $s_w[i]=|w[1..i]|_1$. The \emph{meet} (resp.~\emph{join}) of two binary words $u$ and $v$ is defined as the binary word  $w= u\wedge v$ (resp.~$w= u\vee v$) whose associated sequence is $s_w[i]=\min \{s_u[i],s_v[i]\}$ (resp.~$s_w[i]=\max \{s_u(i),s_v(i)\}$). 
\end{definition}

In other words, $u\wedge v$ (resp.~$u\vee v$) is precisely the maximum lower bound (resp.~minimum upper bound) of the set $\{u,v\}$ with respect to the dominance order. From a geometrical perspective, the meet (resp.~join) of $u$ and $v$ turns out to be the word delimiting the intersection (resp.~union) of the areas below them. 

\begin{proposition}
Let $u$ and $v$ be two digitally convex words. Then  $u\wedge v$ is a digitally convex word.
\end{proposition}

\begin{proof}
Let us consider the (Euclidean) convex hulls $e_u$ and $e_w$ of $u$ and $w$, respectively. We prove that $w=u\wedge v$ is digitally convex since it is the approximation (from below) of the (convex) intersection $e_w=e_u \cap e_v$. We proceed by contradiction assuming that there exists a point $p=(i,j)$ with integer coordinates between $w$ and $e_w$, i.e., $e_w[i]\geq j >w[i]$. Consider w.l.g. that $e_w[i]=\min\{e_u[i],e_v[i]\}=e_u[i]$. Consequently, $\min \{s_u[i],s_v[i]\}=s_u[i]=s_w[i]$, by  definition of meet. Since $u$ is the approximation of $e_u$, it holds $e_u[i]\geq s_u[i]=s_w[i]\geq j$, against the assumption on $p$.
\qed \end{proof}

Hence, the set $\DC\cap\{0,1\}^n$ is a \emph{meet-semilattice} for all $n$, as is the set $\DC_{a,b}$ of all digitally convex words with Parikh vector $(a,b)$, for any given fixed pair $(a,b)$.

As one can expect, the join of two digitally convex words is not in general a digitally convex word, according to the fact that the union of convex polygons does not preserve convexity.  
As an example, consider $u=010101010110001001$ and $v=011110001001001001$. The word $u\vee v=011110001010001001$, whose Lyndon factorization is $w=(01111)(000101)(0001)(0)(0)$ contains a non-balanced Lyndon factor $000101$, so it is not digitally convex by Theorem~\ref{thm:digcon}.

\section{Inflation/Deflation of Digitally Convex Words}\label{sec:infldefl}

Borrowing the terminology from  \cite{lama_yukiko_2023}, we call {\em deflation} the operation that changes a digitally convex word $w=u10v$ into a digitally convex word $w'=u01v$. The name suggests the geometrical interpretation of the operation, i.e., the removal of a point from the (discrete) digitally convex set related to $w$ to obtain a new (discrete) digitally convex set related to $w'$. Similarly, the addition of a point to a digitally convex set preserving the convexity is called {\em inflation}, and it is realized by changing a $01$ occurrence into a $10$ while preserving the convexity of a digitally convex word. 

\begin{lemma}[see \cite{DBLP:conf/acpr/TarsissiCKR19,lama_yukiko_2023}]
\label{lem:good-defl}
    Let $w_1=u1$ and $w_2=0v$ be lower Christoffel words such that $w_1>w_2$. Then $u01v$ is digitally convex.
\end{lemma}
\begin{proof}
    By induction on $|w_1w_2|$. The assertion is clearly verified when $u=v=\varepsilon$, so let us suppose $|w_1w_2|\geq 3$.

    If $w_2$ ends with $0$, then $v=0^i$ for some $i\geq 0$. If $i>0$, then by induction $u010^{i-1}$ is digitally convex, whence so is $u01v$. If $i=0$, then $u01v=u01$ is balanced because $u$ is a right special balanced word, hence a suffix of a central word $u'$, so that $u'01$ is a standard word, whence $u01$ is balanced. 

    The case where $w_1$ begins with $1$ is symmetric. We can therefore assume that both $w_1$ and $w_2$ begin with $0$ and end with $1$. Hence, there exist integers $0\leq m\leq n$ such that $w_1\in\{0^m1,0^{m+1}1\}^*$ and $w_2\in\{0^n1,0^{n+1}1\}^*$; in particular, since $w_1$ and $w_2$ are both Christoffel words, we may assume $w_1$ ends with $0^m1$ and $w_2$ begins with $0^{n+1}1$.
    
    If $m=n$, then $w_1=\varphi_m(w_1')$ and $w_2=\varphi_m(w_2')$, where $w_1',w_2'$ are Christoffel words and $\varphi_m$ is the (Christoffel) morphism such that $\varphi_m(0)=0^{m+1}1$ and $\varphi_m(1)=0^m1$. It is easy to see that since $w_1>w_2$, we have $w_1'>w_2'$. Moreover, our assumption implies that $w_1'=u'1$ and $w_2'=0v'$ for some $u',v'$ such that $u=\varphi_m(u')$ and $v=\varphi_m(v')$. Since $|w_1'w_2'|<|w_1w_2|$, by induction $u'01v'$ is digitally convex, whence so is $\varphi_m(u'01v')=u01v$ (because its Lyndon factorization is just the image under $\varphi_m$ of the one for $u'01v'$).
    
    Let then $m<n$. As observed above, $u01$ is balanced, and clearly so is $v$, being a suffix of $w_2$. Therefore, 
    all elements in the Lyndon factorizations $(u_1,\ldots,u_\ell)$ of $u01$ and $(v_1,\ldots,v_{\ell'})$ of $v$ are Christoffel words.
    Now, $u_\ell$ cannot be a proper suffix of $0^{m+1}1$, since $u_{\ell-1}$ cannot end with $0$; hence, $u_\ell$ has $0^{m+1}1$ as a suffix. Since it is a (power of a) Lyndon word, it cannot begin with $0^j1$ with $j<m+1$. It follows that $u_\ell=(0^{m+1}1)^h$ for some $h\geq 1$. Similarly, it is easy to see that $v_1=(0^n1)^k$ for some $k\geq 1$.

    Therefore, if $m+1<n$ then the factorization $(u_1,\ldots,u_\ell,v_1,\ldots,v_{\ell'})$ of $u01v$ is decreasing, and is therefore \emph{the} Lyndon factorization of $u01v$. Since it is made of Christoffel words, it follows that $u01v$ is digitally convex; if instead $m+1=n$, then the above argument applies to the factorization $(u_1,\ldots,u_{\ell-1},u_\ell v_1,v_2,\ldots,v_{\ell'})$, as $u_\ell v_1=(0^n1)^{h+k}$.
\qed \end{proof}

\begin{lemma}
\label{lem:bad-defl}
    Let $w=u10v$ be a lower Christoffel word. Then $u01v$ is \emph{not} digitally convex.
\end{lemma}
\begin{proof}
    We prove our claim by induction on $|w|$. The statement is easily verified for $|w|\leq 4$, so let us assume $|w|\geq 5$. 
    
    First, suppose  $w$ is primitive, and let $w=u10v=u'v'$, with $u'v'$ being the standard factorization. If $|u'|\leq |u|$, i.e., $u=u'r$ for some $r$, then $r10v=v'$, and by induction $r01v$ is not digitally convex, so that neither is $u01v=u'r01v$. Similarly, if $|v'|\leq |v|$ then the prefix of length $|u'|$ in $u01v$ is not digitally convex.
    
    Hence, we may assume that $u'=u1$ and $v'=0v$. As $u'$ and $v'$ are primitive Christoffel words such that $u'<v'$, we must have $u=0u''$ and $v=v''1$ for some central words $u'',v''$. Moreover, as $|w|>2$ it is well known that the standard factorization $w=u'v'$ is such that either $u'$ is a prefix of $v'$, or $v'$ is a suffix of $u'$.
    \begin{itemize}
    	\item If $u'=0u''1$ is a prefix of $v'=0v''1$, then $u''1$ is a prefix of $v=v''1$, so that $u01v$ begins with $0u''01u''1$, which is not digitally convex by Theorem~\ref{thm:maw-dc}.
	\item Similarly, if $v'$ is a suffix of $u'$ then $u01v$ ends with $0v''01v''1$, not digitally convex by Theorem~\ref{thm:maw-dc}.
    \end{itemize}
    
    Thus, the case of a primitive $w$ is settled; let then $w=(w')^k$ for some $k\geq 2$ and $w'$ a primitive Christoffel word. Similarly to the primitive case, if the occurrence of $10$ in $u10v$ is contained in an occurrence of $w'$, then the assertion follows by induction. 
    
    Suppose that is not the case, i.e., $u1=(w')^j$ and $0v=(w')^{k-j}$ for some $j<k$. Then, letting $w'=0w''1$, the word $u01v$ contains $0w''01w''1$, again not digitally convex in view of Theorem~\ref{thm:maw-dc}.
\qed \end{proof}

\begin{theorem}[see \cite{DBLP:conf/acpr/TarsissiCKR19,lama_yukiko_2023}]
\label{thm:defl}
Let $w=u10v=w_1\cdots w_k$ be a digitally convex word, where $w_1 >\cdots >w_k$ ($k\geq 2$) are Christoffel words. Then $u01v$ is digitally convex if and only if $0v=w_iw_{i+1}\cdots w_k$ for some $1< i\leq k$. 
\end{theorem}
\begin{proof}
If no $i$ is such that $0v=w_i\cdots w_k$, it means that there exists $j\leq k$ such that $w_j=u'10v'$, where $u'$ is some suffix of $u$ and $v'$ is some prefix of $v$. By Lemma~\ref{lem:bad-defl}, $u'01v'$ is not digitally convex, so that neither is $u01v$.

Let us then suppose $0v=w_i\cdots w_k$, and let $u',v'$ be such that $w_{i-1}=u'1$ and $w_i=0v'$. By Lemma~\ref{lem:good-defl}, $u'01v'$ is digitally convex, and so are $w_1\cdots w_{i-2}$ and $w_{i+1}\cdots w_k$ by hypothesis. Hence, to prove that the word\[u01v=w_1\cdots w_{i-2}\cdot u'01v'\cdot w_{i+1}\cdots w_k\] is digitally convex too, it suffices to show that the first and last term of the Lyndon factorization of $u'01v'$, say $z_1$ and $z_\ell$ respectively, satisfy $w_{i-1}> z_1$ and $z_\ell> w_i$.

Indeed, either $z_1$ is a prefix of $u'$, and so a proper prefix of $w_{i-1}$, or it begins with $u'0<u'1=w_{i-1}$. Symmetrically, $\widehat z_\ell$ (that is, the complement of $\widetilde z_\ell$) must be either a proper prefix of $\widehat w_i$, or begin with $\widehat v'0<\widehat v'1=\widehat w_i$. Hence, $\widehat z_\ell$ has a lower slope than $\widehat w_i$, so that $z_\ell$ has a higher slope than $w_i$.
\qed \end{proof}

\begin{remark}
    Theorem~\ref{thm:defl} shows that the deflation of a digitally convex word $w$ 
    is a \emph{local} operation, in that all elements of the Lyndon factorization of $w$ are inherited in the Lyndon factorization of $w'$, except for the two elements that overlap the given occurrence of $10$. By contrast, the converse inflation operation can change an arbitrarily large number of elements in the Lyndon factorization, as shown in the next example.
\end{remark}

\begin{example}
    Let $f=0100101001001\cdots$ be the Fibonacci infinite word, fixed point of the substitution $0\mapsto 01, 1\mapsto 0$, and consider the Sturmian word $s=001f$. As shown in~\cite{DBLP:conf/stacs/Melancon96}, the Lyndon factorization of $f$ is \[\prod_{n\geq 1}\ell_{2n+1}=(01)( 00101)(0010010100101)\cdots\]
    where $\ell_1=1$, $\ell_2=0$, $\ell_{2n+1}=\ell_{2n}\ell_{2n-1}$, and $\ell_{2n+2}=\ell_{2n}\ell_{2n+1}$ for $n\geq 1$ give the sequence of all lower Christoffel factors of $f$ (note that $|\ell_n|=F_n$, the $n$th Fibonacci number).
    It follows that for all $k>1$, the Lyndon factorization of the prefix $p_k$ of $s$ of length $3+\sum_{n=1}^kF_{2n+1}=F_{2k+2}+2$ is $p_k=(00101)^2\cdot\prod_{n=3}^k\ell_{2n+1}$.
    
    Now, the inflated infinite word $s'=010f$ is still Sturmian, so that its prefixes are all digitally convex; in particular, for all $k>1$,
    the Lyndon factorization of its prefix $p_k'$ such that $|p_k'|=|p_k|$ is $p_k'=01\cdot \ell_{2k+2}$, thus showing that an arbitrarily large number of elements in the Lyndon factorization of a digitally convex word can give rise to a constant number (2, in this case) of such elements after inflation.
\end{example}

In the next lemma we analyze how the inflation operation acts on a single Christoffel word.

\begin{lemma}
    \label{lem:inflat}
    Let $w=V01U$ be a primitive lower Christoffel word. Then $V10U$ is digitally convex if and only if $w=0UV1$, i.e., $(V0,1U)$ is the palindromic factorization of $w$ and $(0U,V1)$ is its standard factorization.
\end{lemma}
\begin{proof}
    By induction on $|w|$. If $w=01$, so that $U=V=\epsilon$, then $10$ is digitally convex and there is nothing to prove.

    Let then $|w|>2$, and $w=uv$ be its standard factorization. As $w\neq 01$, either $u$ ends with $1$, or $v$ begins with $0$, or both; therefore, any occurrence of $01$ in $w$ lies completely within $u$ or within $v$. Moreover, either $u$ is a prefix of $v$, or $v$ is a suffix of $u$.
    Let us assume the former case holds, as the latter one is similar; let then $v=uv_1$ for some primitive Christoffel word $v_1$. By induction, if $v=V'01U'$ and $V'10U'$ is digitally convex, then $0U'=u$ and $V'1=v_1$. Thus, if $V=uV'$ and $U=U'$, then $V10U=uv_1u=vu$ is digitally convex.
    
    It remains to show that no occurrence of $01$ within $u$ works, i.e., that if $|V01|\leq |u|$, then $V10U$ is not digitally convex. Indeed, let $n\geq 1$ be maximal such that $u^n$ is a prefix of $v$, and write $v=u^nv_n$ for some primitive Christoffel word $v_n$. If $uv_n=01$, then clearly $u=0$ has no occurrences of $01$; otherwise, by maximality of $n$, $v_n$ must be a suffix of $u$, so that $u$ has standard factorization $u=u_1v_n$ for some $u_1$. By induction, if $u=V'01U'$ and $V'10U'$ is digitally convex, then $0U'=u_1$ and $V'1=v_n$. Hence, for $V=V'$ and $U=U'v$, the word
    \[V10U=v_nu_1v=v_nu_1u^nv_n=v_nu_1(u_1v_n)^nv_n\]
    is not digitally convex by Theorem~\ref{thm:maw-Prov}, as $u=u_1v_n$ is a Christoffel word and $n\geq 1$.
\qed \end{proof}

The previous lemma essentially states that the only possible inflation point in a Christoffel word is the point $S'$ determined by the palindromic factorization, which is the  point on the path encoding the Christoffel word that is the farthest from the Euclidean segment (see Fig.~\ref{fig:dec}).
For example, let $w=0010010\mathbf{01}01=001\cdot 00100101$. Then the inflated word is $0010010\mathbf{10}01=00100101\cdot 001$. 

A deeper investigation reveals that the inflation operation performed as in  Lemma~\ref{lem:inflat}  cannot be applied to any Christoffel factor of a digitally convex word while preserving the digital  convexity property. The following example is from \cite{DBLP:conf/cwords/DulioFRTV17}.

\begin{example}\label{ex:bad_inflation}
Let $w$ be the digitally convex word $w=w_{5,3}w_{20,11}$. The application of the inflation to $w_{5,3}$, according to Lemma~\ref{lem:inflat}, produces $w'=w_{3,2}w_{2,1}w_{20,11}$ that is not digitally convex any more. In order to gain back the digital convexity, it is required a second inflation in $w_{20,11}$, so obtaining $w''=w_{3,2}w_{2,1}w_{9,5}w_{11,6}$. Note that $w_{2,1}w_{9,5}$ is the standard factorization of $=w_{11,6}$; so, finally,  $w''=w_{3,2}w_{11,6}w_{11,6}$. 
\end{example}

\begin{remark}
    Example~\ref{ex:bad_inflation} suggests that an order on the inflation operations can be established to avoid the loss of the digital convexity property. In fact, performing on $w$ a first inflation in its Christoffel factor $w_{20,11}$ produces the digitally convex word $w_{5,3}w_{9,5}w_{11,6}$. Now, the second inflation on $w_{5,3}$ leads to $w''$.  
\end{remark}

The following lemma expresses what observed in the previous remark.

\begin{lemma}
    \label{lem:infla-dc}
    Let $w=w_1\dots w_k$ be a digitally convex word.  where $w_1 >\cdots >w_k$ ($k\geq 2$) are Christoffel words. There exists an index $1\leq i \leq k$ such that $w_i$ has palindromic factorization $(u0,1v)$ and $w'=w_1\dots w_{i-1}\: u10v \: w_{i+1}\dots w_k$ is digitally convex. 
\end{lemma}
\begin{proof}
  Let us proceed by contradiction assuming that no such index $i$ exists. This implies that each possible inflation on the $S'$ point of a Christoffel factor $w_i$ produces a non-digitally convex word. We use the implication $i\rightarrow j$ each time the inflation of the $S'$ point of $w_i$ requires at least the inflation of the $S'$ point of $w_j$ to gain back the digital convexity property of the generated word. By hypothesis and by the finiteness of the Christoffel factors of $w$, there exists at least a cycle of implications. Let us consider two indexes $i$ and $j$ of the cycle, with $i<j$. It holds that $i \xrightarrow{*} j$ and $j\xrightarrow{*} i$, where $\xrightarrow{*}$ indicates the transitive closure of $\rightarrow$. Let $u$ (resp.~$v$) be the convex hull generated by the applications of inflation on $w_i$ (resp.~$w_j$). It holds that the Christoffel factor of $u$ starting from the $S'$ point of $w_i$ and ending at the beginning of a Christoffel factor, possibly empty, $w_t$ of $w$, with $t>j$, includes the $S'$ point of $w_j$. Symmetrically the Christoffel factor of $v$ starting at the end of a Christoffel factor $w_s$, possibly empty, of $w$, with $s\leq i$, and ending in the $S'$ point of $w_j$ includes the $S'$ point of $w_i$. As a consequence, the convex hull including $w_t$ and $w_s$ must include the $S'$ points of $w_i$ and $w_j$ against the digital convexity of $w$.    
  \qed \end{proof}

\begin{proposition}\label{prop:deflation}
    Let $w$ be a digitally convex word with Parikh vector $(a,b)$. There exists a sequence of applications of inflation that leads from the word $w_{a,b}$ to $w$.   
\end{proposition}

\begin{proof}
Let $w_1\cdots w_k$ be the Lyndon factorization of $w$, where $w_1 >\cdots >w_k$ ($k\geq 2$) are Christoffel words. Lemma~\ref{lem:good-defl} assures that each deflation between two consecutive words $w_i w_{i+1}$, with $1\leq i \leq k-1$, produces a digitally convex word. Choosing randomly a deflation point and iterating the process, after a finite number of steps we obtain a digitally convex word with Parikh vector $(a,b)$ and whose Lyndon decomposition has one single Christoffel word, i.e., $w_{a,b}$. Reversing the applications of deflation steps from $w_{a,b}$ to $w$ we obtain the desired sequence of applications of inflation.    
\qed \end{proof}

\begin{proposition}\label{prop:inflation}
    Let $w$ be a digitally convex word with Parikh vector $(a,b)$. There exists a sequence of applications of deflation that leads from the digitally convex word $1^b0^a$ to $w$.   
\end{proposition}

\begin{proof}
  Acting as in Proposition~\ref{prop:deflation}, let $w_1\cdots w_k$ be the Lyndon factorization of $w$, where $w_1 >\cdots >w_k$ ($k\geq 2$) are Christoffel words. Lemma~\ref{lem:inflat} assures that there exists an inflation in $w$ that generates a digitally convex word $w'$. Iterating the process, after a finite number of steps we obtain the digitally convex word $1^b0^a$,  with Parikh vector $(a,b)$ and where no $01$ sequences are present. Reversing the applications of inflation steps from $1^b0^a$ to $w$ we obtain the desired sequence of applications of deflation. 
\qed \end{proof}

We therefore arrive at the main result of this section, whose proof directly follows from Propositions~\ref{prop:deflation} and \ref{prop:inflation}.

\begin{theorem}\label{teo:infdef}
Iterated inflation (resp.~deflation) in the word $w_{a,b}$ (resp.~$1^b0^a$) produces all digitally convex words with the same Parikh vector $(a,b)$.
\end{theorem}

Now, we are ready to define the order relation $\leq_I$ on the set $\DC_{a,b}$ of all the digitally convex words with Parikh vector $(a,b)$ such that, provided $u,v \in \DC_{a,b}$, $u\leq_I v$ if there exists a sequence of applications of the inflation operation leading from $u$ to $v$. It is worthwhile proving that the structure $\mathcal{W}_{a,b}=(\DC_{a,b},\leq_I)$ is a partial order, having the words $1^b0^a$ and $w_{a,b}$ as maximum and minimum element, respectively.

Let us indicate the partial order provided by the dominance order on the same set $\DC_{a,b}$ as $\DOM_{a,b}=(\DC_{a,b},\sqsubseteq)$.

\begin{theorem}
The partial orders $\mathcal{W}_{a,b}$ and $\mathcal{D}_{a,b}$ define the same structure on the ground set $\DC_{a,b}$. 
\end{theorem}

\begin{proof}
Let $u$ and $v$ be two digitally convex words in $W_{a,b}$. Assume that $v$ covers $u$ in the $\leq_I$ order, i.e., $v$ is obtained by $u$ with a single inflation performed, w.l.g., on the Christoffel factor $w_i=w'01w''$ of $u=w_1\dots w_i \dots w_k$. The word $v$ turns out to be $w_1 \dots w'10w'' \dots w_k$ and it holds $u\leq_D v$.

Now, assume that $u$ covers $v$ in the $\leq_D$ order, so $u$ and $v$ differ in two indexes only, say $i$ and $j$, with $i<j$. By definition of dominance order, it holds $u[i]=0$, $v[i]=1$, $u[j]=1$, $v[j]=0$. 

If $j=i+1$, then $v$ is obtained by an inflation on $u$, so $u\leq_I v$.

Otherwise there exists an index $t$ such that $i<t<j$. Two cases arise:
\begin{description}
\item{$u[t]=v[t]=0$: } there exists $w=u[1\dots t-1]\:1\:v[t+1\dots k]$ that lies between $u$ and $v$ in the dominance order;
\item{$u[t]=v[t]=1$: } there exists $w=v[1\dots t-1]\:0\:u[t+1\dots k]$ that lies between $u$ and $v$ in the dominance order.
\end{description}

In both cases we reach a contradiction with the assumption that $v$ covers $u$.

So when $v$ covers $u$ in the dominance order, we also have $u\leq_I v$.
Transitivity leads to the thesis.   
\qed \end{proof}

 \section{Acknowledgments}

 We warmly thank Lama Tarsissi for useful discussions on Digitally Convex words.

\bibliographystyle{abbrv}

\end{document}